\documentclass[12pt]{amsart}

\usepackage{amsmath}
\usepackage{amsfonts}
\usepackage{amssymb}
\usepackage{amsthm}
\usepackage[english]{babel}
\usepackage[latin1]{inputenc}
\usepackage[T1]{fontenc}
\usepackage{graphicx}

\usepackage{geometry}
\newtheorem{theo}{Theorem}[]
\newtheorem{prop}[theo]{Proposition}
\newtheorem{lemme}[theo]{Lemma}
\newtheorem{definition}[theo]{Definition}
\newtheorem{coro}[theo]{Corollary}
\newtheorem{remarque}[theo]{Remark}

\DeclareMathOperator{\diverg}{div}
\DeclareMathOperator{\Hess}{Hess}

\begin{document}

\begin{abstract}
Using a maximum principle for self-shrinkers of the mean curvature flow, we give new proofs of a rigidity theorem for rotationally symmetric compact self-shrinkers and a result about the asymptotic behavior of self-shrinkers. This comparison argument also implies a linear bound for the second fundamental form of self-shrinking surfaces under natural assumptions. As a consequence, translating solitons can be related to these self-shrinkers.
\end{abstract}

\title [A maximum principle for self-shrinkers] {A maximum principle for self-shrinkers and some consequences}
\author{Antoine Song}
\date{}
\maketitle

\begin{center}
\textsc{\small{}}
\end{center}

\section*{}
	A hypersurface $\Sigma \subset \mathbb{R}^{n+1}$ is a self-shrinker for the mean curvature flow, normalized so that it shrinks in unit time to the origin $0\in\mathbb{R}^{n+1}$, if it satisfies the following equation:
\begin{equation} \label{self}
H = \frac{\langle x , \nu \rangle}{2}
\end{equation}
where $\nu$ is the outer normal unit vector, $H = \diverg (\nu)$ denotes the mean curvature and $x$ is the position vector. Such a surface will be called more briefly a self-shrinker. Equivalently, one can say that the set of hypersurfaces $\{\sqrt{-t} \Sigma ; t<0 \}$ is a solution of the mean curvature flow, i.e. verifies the equation :
\begin{equation} \label{flot}
(\partial_t x)^{\bot}=-H \nu.
\end{equation}

	Self-shrinkers provide models for blow-ups at singularities of mean curvature flow: consider a family of hypersurfaces evolving by the mean curvature flow and starting from a closed embedded hypersurface and focus on a point of singularity, then rescalings yield a subsequence converging weakly to a "tangent flow", which satisfies (\ref{self}) (see \cite{Ilmanen2}, \cite{White}, \cite{CM1}). The classification of embedded self-shrinkers proved to be a difficult problem. The simplest examples are given by cylinders $\mathbb{S}^k \times \mathbb{R}^{n-k}$ where $\mathbb{S}^k$ is the $k$-sphere of radius $\sqrt{2k}$. If $n\geq 2$, Huisken, then Colding and Minicozzi showed that those hypersurfaces were the only ones with polynomial volume growth and whose mean curvature $H$ is nonnegative (\cite{Huisken1}, \cite{Huisken2}, \cite{CM1} Theorem $0.17$). When $n=1$, straight lines passing through the origin and the circle of radius $\sqrt{2}$ are the only self-shrinking embedded curves \cite{AbreschLanger}, but as soon as the dimension is greater than or equal to $2$, there are non trivial self-shrinkers, as the Angenent torus \cite{Angenent}. See also \cite{KKM} and \cite{Nguyen} for construction of complete embedded self-shhrinkers of high genus.
	
	The aim of this note is to present some applications of a maximum principle adapted to self-shrinking hypersurfaces viewed locally as graphs (Proposition \ref{PMsimple}, Corollary \ref{comparaison}). It will be used to rule out some hypersurfaces from the set of self-shrinkers. The usual maximum principle states for example that the distance between two compact hypersurfaces moving by the mean curvature flow is non-decreasing in time. Here, we prove a maximum principle for graphs, the advantage being that within a hypersurface, sometimes one can find two subsets forming graphs of two functions whose difference achieves a minimum, whereas a minimal distance (even a local one) between the two subsets is not achieved. 
	
	Using this remark (see Proposition \ref{PMsimple}, Corollary \ref{comparaison}), we prove that the only embedded rotationally symmetric compact self-shrinkers are either a $\mathbb{S}^1 \times \mathbb{S}^{n-1}$ or a round sphere with the appropriate radius (this fact is part of a more general statement covering the non compact case proven by Kleene and M{\o}ller \cite{KleeneMoller}):

\begin{theo} [\cite{KleeneMoller}]
Let $\Sigma^n \subset \mathbb{R}^{n+1}$ be a compact embedded rotationally symmetric hypersurface. If $\Sigma$ is a self-shrinker then either it is the sphere of radius $\sqrt{2n}$ centered at the origin or a $\mathbb{S}^1 \times \mathbb{S}^{n-1}$.
\end{theo}

 Then, with the same argument, we show that self-shrinkers are weakly asymptotic to cones, a result known by Ilmanen (\cite{Ilmanen} p.8, though I didn't find his proof). In the two-dimensional case, some additional information is given.

 \begin{theo} [\cite{Ilmanen}]
 If $\Sigma \subset \mathbb{R}^{n+1}$ is a complete properly immersed self-shrinker, then there exists a cone $\mathcal{C}$ such that $$\lambda \Sigma \to \mathcal{C} \text{  when  } \lambda \to 0^+$$
locally in the Hausdorff metric on closed sets.
 
 In the case $n=2$, if the number of connected components of $S(0,r) \cap \Sigma$ is bounded above, then $\mathcal{C} \cap \mathbb{S}^2$ has 2-dimensional Lebesgue measure $0$.
 \end{theo}

Finally in the case of surfaces, our result in Section \ref{principale} gives a linear bound for the norm of the second fundamental form $|A|$ of self-shrinking surfaces, under the assumption that the "curvature concentration" is bounded, i.e. that these surfaces satisfy a local integral bound for $|A|$ (see Definition \ref{concentration}). This is a natural assumption in the context of weak blowups (see Proposition \ref{naturality}). It is proved that:
 
 \begin{theo}
If $\Sigma \subset \mathbb{R}^{3}$ is a complete properly embedded self-shrinker of finite genus $g$ such that "the curvature concentration is bounded by $\kappa$", then
 $$\exists C, \forall x \in \Sigma, |A(x)| \leq C(1+ |x|).$$ 

Moreover for such a surface with ends that are "$\delta$-separated at infinity", the constant $C$ only depends on $g$, $\delta$, $\kappa$ and on a "bound for the topology of $\Sigma$".

 \end{theo}

For more precise statements of this result, see Definition \ref{separate}, Definition \ref{bounded}, Theorem \ref{MAIN} and Theorem \ref{main}.

As in \cite{CM2}, the surfaces in the second part of the theorem are supposed to be homeomorphic to closed surfaces with finitely many disjoint disks removed. The proof is based on a blow-up argument, the compactness for self-shrinkers \cite{CM2}, the maximum principle and some results about classical minimal surfaces. Consequently, it is pointed out that translating solitons of the mean curvature flow can model regions of $\Sigma$ far from the origin (Corollary \ref{translating}).

\subsection*{Acknowledgement} 
I would like to thank Niels M{\o}ller for proofreading a preliminary version of this note and for his numerous suggestions. I also wish to thank my professors Fernando C. Marques, Olivier Biquard and Laurent Hauswirth for their guidance.

 \vspace{1cm}

 \section{A maximum principle for embedded self-shrinkers}
 
 The mean curvature flow equation for graphs is given by (\cite{Mantegazza} p.$10$):
 
 \begin{lemme}
 Suppose that $\phi_t : M \to \mathbb{R}^{n+1}$ are smooth hypersurfaces moving by mean curvature and are graphs on the open subset $\Omega$ of the hyperplane $\langle e_1 , ... , e_n \rangle \subset \mathbb{R}^{n+1}$, that is, there exists a smooth function $f : \Omega \times [0,T) \to \mathbb{R}$ with $$\phi_t(p) = (x_1(p), ... , x_n(p), f(x_1(p), ... , x_n(p), t)),$$ then

\begin{equation} \label{graphe}
\partial_t f = \Delta f - \frac{\Hess f (\nabla f , \nabla f)}{1 + |\nabla f|^2} = \sqrt{1+ |\nabla f|^2} \diverg ( \frac{\nabla f}{\sqrt{1+ |\nabla f|^2}} ) .
\end{equation}

Conversely, if $f$ satisfies this equation, then it corresponds to hypersurfaces moving by mean curvature.

 \end{lemme}

Next we give the maximum principle for self-shrinkers:

\begin{lemme} \label{principe du max}
Let $f,g : \Omega \times [0,T) \subset{\mathbb{R}^n \times \mathbb{R}} \to \mathbb{R}$ be two functions satisfying (\ref{graphe}), where $\Omega$ is an open subset of $\mathbb{R}^{n}$. Suppose that there exists a compact set $K \subset \Omega$ such that for all $t' \in [0,T)$, the minimum of $(f-g)(.,t')$ is attained at least at one point of $K$. 

Define $u(t) = \min_{p \in \Omega} (f-g)(p,t)$. Then $u$ is a locally Lipschitz function, hence differentiable almost everywhere and if it exists, the differential is nonnegative.

\end{lemme}

\begin{proof}
By Hamilton's trick (\cite{Mantegazza} p.26), $u$ is a locally Lipschitz function, hence differentiable almost everywhere and where it makes sense:
$$\frac{du(t)}{dt} = \frac{\partial (f-g)(p,t)}{\partial t},$$
$p\in K$ being a point where the minimum of $(f-g)$ is attained. But at such a point $p$, $\nabla f = \nabla g$, the Hessian of $f-g$ is nonnegative and $\frac{\nabla f}{\sqrt{1+|\nabla f|^2}}$ is a vector whose Euclidian norm is less than $1$, hence the lemma.

\end{proof}

To show that a hypersurface $\mathcal{S}$ is not a self-shrinker, one can thus try to apply this maximum principle to two graphs given by two parts of $\mathcal{S}$. Next we give a proposition implementing this strategy and which will be used in various forms throughout this note:

\begin{prop} \label{PMsimple}
Let $\mathcal{S}_1$ and $\mathcal{S}_2$ two disjoint complete hypersurfaces of $\mathbb{R}^{n+1}$, which can be noncompact and have boundaries. Define the function $h : \mathcal{S}_1 \to \mathbb{R} \cup \{ \infty \}$:
\begin{align*} 
h(a) = &\min\{ a_{n+1} -  b_{n+1} ; (a_1, ... , a_{n}, b_{n+1}) \in \mathcal{S}_2 \text{ and } b_{n+1}<a_{n+1}\}
\\ & \text{ if } \{b_{n+1} ; b_{n+1}<a_{n+1} \text{ and }(a_1, ... , a_{n}, b_{n+1}) \in \mathcal{S}_2 \} \neq \varnothing \\
h(a) = & \infty  \text{ otherwise}.
\end{align*}

If $h$ achieves a local finite minimum at $a \in \mathcal{S}_1$ and $b \in \mathcal{S}_2$, two points not in the respective boundaries, and if $\langle \nu_a , e_{n+1} \rangle \neq 0$ where $\nu$ is the outward normal vector of $\mathcal{S}$, then $\mathcal{S}_1$ or $\mathcal{S}_2$ is not a part of a self-shrinker.
\end{prop}

\begin{proof}
Suppose that the two hypersurfaces are self-shrinkers. Under these hypotheses, we can locally write $\mathcal{S}_1$ and $\mathcal{S}_2$ as graphs of $f$ and $g$ which satisfy the hypotheses of lemma \ref{principe du max}. If $u(t) = \min (f-g)(p,t)$ then $\frac{d u_{\min}(t)}{dt} \geq0$ almost everywhere. On the other hand, by the definition of self-shrinker, the solution $\mathcal{S}_i(t)$ ($i=1,2$) of the mean curvature flow with $\mathcal{S}_i(0)=\mathcal{S}_i$ is given by $\mathcal{S}_i(t) = \sqrt{-t} \mathcal{S}_i$ $(-1 \leq t \leq0)$, which would mean $\frac{d u_{\min}(t)}{dt} < 0$. This is the desired contradiction.
\end{proof}

Next, we give a specialized form of this proposition, but before let's define hypersurfaces with $\delta$-separated ends.
 
 \begin{definition} \label{separate}
Let $\delta>0$. $\mathcal{S} \subset \mathbb{R}^{n+1}$ is a properly immersed hypersurface. 
Suppose that $\mathcal{S}_1$ and $\mathcal{S}_2$ are two disjoint open connected subsets of $\mathcal{S}$ with $\overline{\mathcal{S}_1} \cup \overline{\mathcal{S}_2} = \mathcal{S}$ $(\overline{\mathcal{S}_i}$ being the closure of ${\mathcal{S}_i}$ in $\mathcal{S}$) and $\partial \overline{\mathcal{S}_1} = \partial\overline{\mathcal{S}_2}$. Suppose that this boundary $\partial\overline{\mathcal{S}_1}$ is bounded. 

If for every such partition $\overline{\mathcal{S}_1} \cup \overline{\mathcal{S}_2} = \mathcal{S}$, for $r$ sufficiently large, 
one of the $(\frac{1}{r}  \mathcal{S}_i)  \backslash B(0,1)$ is empty or
$$\inf \{ ||x_1-x_2|| ; x_i \in (\frac{1}{r}  \mathcal{S}_i)  \backslash B(0,1)  \} \geq \delta$$
then $\mathcal{S}$ is said to have $\delta$-separated ends.
\end{definition}
 
\begin{coro} \label{comparaison}
Let $\delta >0$. Let $\mathcal{S} \subset \mathbb{R}^{n+1}$ be a complete properly immersed hypersurface with $\delta$-separated ends.

Suppose that there exists an embedding $\alpha : \mathbb{S}^{n-1} \to \mathcal{S}$ such that the two disjoint connected components of $\mathcal{S} \setminus \alpha(\mathbb{S}^{n-1})$ are called $\mathcal{S}^+$ and $\mathcal{S}^-$, and a set $A\subset \mathbb{S}^n$ of non zero $n$-dimensional Lebesgue measure with the following properties: if $\tilde{e}=(e_1, ... ,e_{n+1})$ is an orthonormal base of $\mathbb{R}^{n+1}$ with $e_{n+1} \in A$ then in the coordinates determined by $\tilde{e}$

\begin{enumerate}
\item there exists $x = (x_1, ... , x_n, x_{n+1}) \in \mathcal{S}^-$ and $y = (x_1, ... , x_n, y_{n+1}) \in \mathcal{S}^+$ such that $y_{n+1} < x_{n+1}$,

\item there exists a neighborhood $\mathcal{V}$ of $\alpha(\mathbb{S}^{n-1})$ in the closure of $\mathcal{S}^-$ verifying that if $ a = (z_1, ... , z_n, a_{n+1}) \in \mathcal{V}$ then there is a point $ u=(z_1, ... , z_n, u_{n+1}) \in \mathcal{S}^-$ with $a_{n+1} > u_{n+1}$ such that there is no $v = (z_1, ... , z_n, v_{n+1}) \in \mathcal{S}^+$ with $a_{n+1} > v_{n+1} > u_{n+1}$,

\item if $ z = (z_1, ... , z_{n+1}) \in \alpha(\mathbb{S}^{n-1})$ and $ b = (z_1, ... , z_n, b_{n+1}) \in \mathcal{S}^-$ satisfy $b_{n+1} > z_{n+1}$, then there is a  $ w=(z_1, ... , z_n, w_{n+1}) \in \mathcal{S}^+$ such that $b_{n+1} >w_{n+1} > z_{n+1}$.
\end{enumerate}
Then $\mathcal{S}$ is not a self-shrinker.

\end{coro}

\begin{remarque}
This corollary covers the case where $\mathcal{S}$ is compact (see Figure \ref{configurations}) because such a surface automatically has $\delta$-separated ends. The existence of $A$ will always be easy to check in our applications. Besides, although we will only need the situation of the corollary, where the common boundary is given by an $(n-1)$-sphere $\alpha(\mathbb{S}^{n-1})$, one can actually consider more general boundaries.
\end{remarque}

\begin{proof} 
To prove this corollary, we find two parts of $\mathcal{S}$, the first one in $\mathcal{S}^+$, the other one in $\mathcal{S}^-$, to which is applied the Maximum Principle. Then the conclusion will follow by Proposition  \ref{PMsimple}.
 
 Suppose that $\mathcal{S}$ is a self-shrinker.
 As in Proposition \ref{PMsimple}, define $h : \mathcal{S}^-\cup \alpha(\mathbb{S}^{n-1})  \to \mathbb{R\cup \{ \infty \}}$ by 

 \begin{align*} 
h(a) = &\min\{ a_{n+1} -  b_{n+1} ; (a_1, ... , a_{n}, b_{n+1}) \in \mathcal{S}^+ \cup \alpha(\mathbb{S}^{n-1}) \text{ and } b_{n+1}<a_{n+1}\}
\\ & \text{ if } a \in \mathcal{S}^-  \mbox{ and if such a $b$ exists, } \\
h(a) = & \infty  \text{ if }  a \in \mathcal{S}^- \text{ and if such a $b$ doesn't exist, } \\
h(a) = &  \lim\inf_{n \to \infty} \{h(x) ; x \in \mathcal{S}^- \text{ and } |x-a|<\frac{1}{n} \} \\  & \text{ if } a \in  \alpha(\mathbb{S}^{n-1}).
\end{align*}
 
By Proposition \ref{PMsimple}, we just need to show that $h$ attains a finite minimum at $a \in \mathcal{S}^-$ and $b \in \mathcal{S}^+$ with $\langle \nu_a , e_{n+1} \rangle \neq 0$.
 
This function is not constantly $\infty$ by the first condition and is lower semicontinuous: it attains its finite minimum on the set $\mathcal{S}^-\cup \alpha(\mathbb{S}^{n-1})$ because of the $\delta$-separation hypothesis. In fact, it can't be achieved in $\alpha(\mathbb{S}^{n-1})$ because of the second condition. So we can find $a=(a_1, ... , a_{n+1}) \in \mathcal{S}^-$ and $b=(b_1, ... , b_{n+1}) \in \mathcal{S}$ such that $a_{n+1} - b_{n+1} = \min h$. The third condition gives $b \in \mathcal{S}^+$. Moreover $\nu_a = \nu_b$. Then, the Sard Lemma and the fact that the three conditions are true for $e_{n+1} \in A$ imply that we can suppose $\langle \nu_a , e_{n+1} \rangle \neq 0$: indeed the set of vectors $e\in \mathbb{S}^n$ for which there exists $a \in \mathcal{S}^-$ and $b \in \mathcal{S}^+$ with $\nu_a = \nu_b$ and $\langle \nu_a , e_{n+1} \rangle = 0$ is of Lebesgue measure zero (they are critical values of $(a,b) \in \mathcal{S}^- \times \mathcal{S}^+ \mapsto (a-b)/||a-b||$). The conclusion follows from Proposition \ref{PMsimple}.

\end{proof}

\begin{figure}
\includegraphics[trim = 0mm 110mm 0mm 0mm, clip, width=12cm]{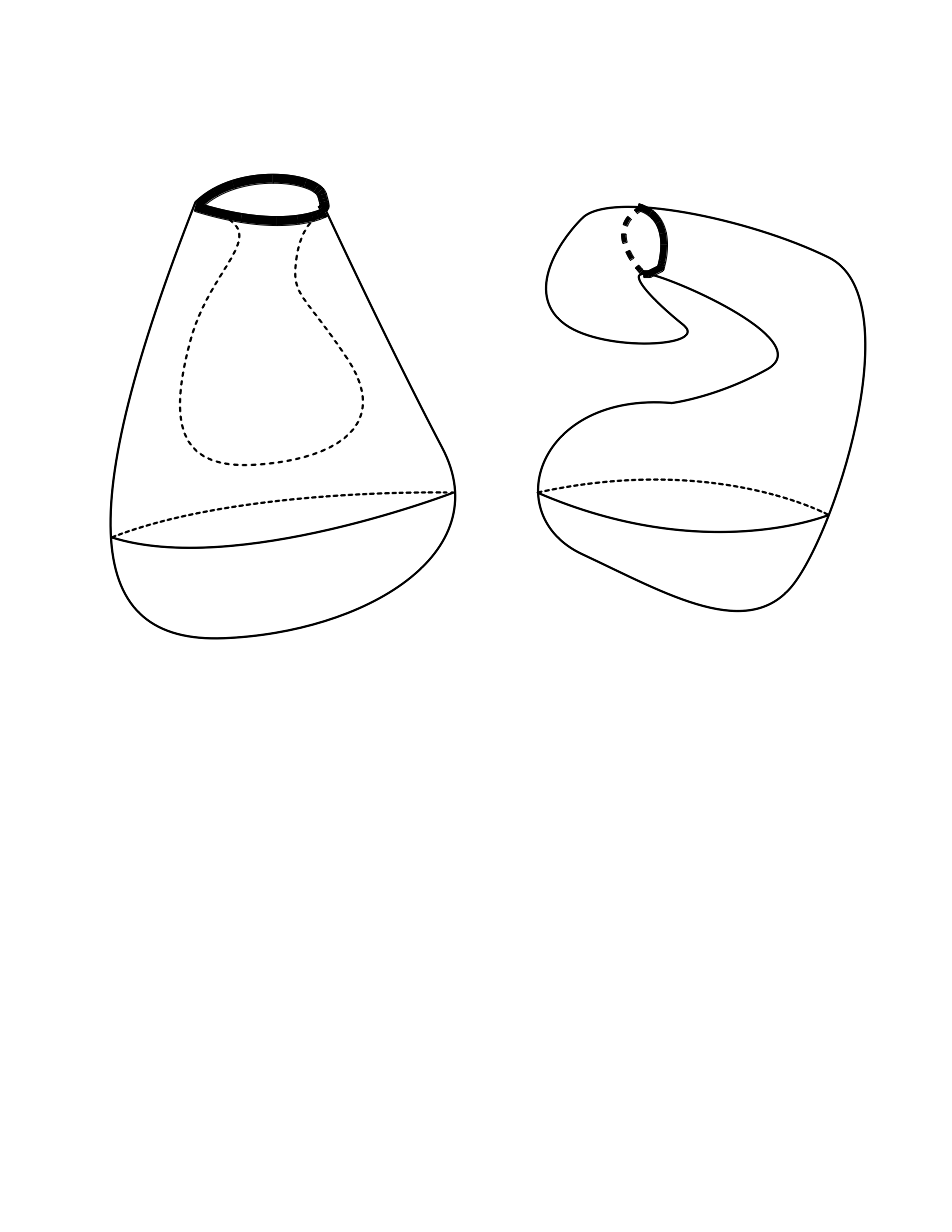}
\caption{Some typical configurations where Corollary \ref{comparaison} applies, but where the usual maximum principle is inefficient. Embeddings of $\mathbb{S}^{n-1}$ are represented in bold line.}

\label{configurations}
\end{figure}

 \section{Some applications}

By "maximum principle for self-shrinkers", one usually means that the distance between two hypersurfaces moving by mean curvature is non-decreasing. Here, Corollary \ref{comparaison} gives a maximum principle for graphs and one has to choose the axis $\mathbb{R}e_{n+1}$: this non canonical choice enables more flexibility, as we will see with the following paragraphs.

\subsection{Rotationally symmetric compact self-shrinkers}

Let $u$ be a vector of $\mathbb{R}^{n+1}$, $\mathbb{S}^{n-1}$ is identified with the unit sphere of the hyperplane orthogonal to $u$. Consider a simple curve $\gamma : [a,b] \to \mathbb{R}\times \mathbb{R}^+$. Let $\Sigma_\gamma$ be the image of an embedding $\varphi : \mathbb{S}^{n-1} \times  [0,1] \to \mathbb{R}^{n+1}$ which can be written as $\varphi(\omega, s) = x(s) u + r(s) \omega $, where $s \mapsto (x(s),r(s))$ is a parametrization of $\gamma$. $\Sigma_\gamma$ is then said to be the rotationally symmetric hypersurface generated by $\gamma$: it is obtained by rotating $\gamma$ around the axis $\mathbb{R}u$. Here is a result proved in \cite{KleeneMoller} (see also \cite{Drugan}).

\begin{theo} [\cite{KleeneMoller}]
If $\Sigma^n \subset \mathbb{R}^{n+1}$ is an embedded compact rotationally symmetric self-shrinker, then $\Sigma^{n}$ is:
\begin{enumerate}
\item either the sphere $\mathbb{S}^n$ of radius $\sqrt{2n}$ centered at the origin,
\item or an embedded $\mathbb{S}^1 \times \mathbb{S}^{n-1}$.
\end{enumerate}
\end{theo}

\begin{proof}
Suppose that $\Sigma^n$ is generated by $\gamma : [a,b] \to \mathbb{R} \times \mathbb{R}^+$. By compactness and embeddedness of $\Sigma$, $\gamma$ is the disjoint union of simple closed curves of $\mathbb{R} \times \mathbb{R}_*^+$ and simple curves whose ends are in $\mathbb{R} \times \{0\}$. The usual maximum principle implies that this union has only one element.

If $\gamma$ is a simple closed curve in $\mathbb{R} \times \mathbb{R}_*^+$ then $\Sigma$ is a torus $\mathbb{S}^1 \times \mathbb{S}^{n-1}$.

Suppose now that $\gamma(0)$ and $\gamma(1)$ are in $\mathbb{R} \times \{0\}$. $\Sigma$ is necessarily diffeomorphic to a sphere and we want to show that it is in fact the sphere of radius $\sqrt{2n}$ centered at the origin. Because of results proven in \cite{Huisken1}, we just have to show that $H\geq 0$ on $\Sigma$. Suppose that this is not the case: $\{s ; H(\gamma(s))<0 \} \neq \varnothing$. $u$ is the horizontal unit vector (identified with $(1,0) \in \mathbb{R} \times \mathbb{R}^+$), let $v$ be the vertical unit vector $(0,1) \in \mathbb{R} \times \mathbb{R}^+$. Write $(x(s),r(s))$ for $(x(\gamma(s)),r(\gamma(s)))$. Denote by $\theta(s)$ the angle between $(x(s),r(s))$ and $\gamma'(s)$ for all the points $\gamma(s)$ different from the origin. The mean curvature $H$ vanishes at $\gamma(s)$ if and only if $\theta(s) =0 [\pi]$ or $x(s)=r(s)=0$ and its sign is given by the sign of $\sin(\theta(s))$. By changing the parametrization, we can suppose that $x(0) > x(1)$. $\Sigma$ being smooth, $\gamma'(0)$ is parallel to $v$ and $\nu(\gamma(0))$ is parallel to $u$.

\begin{lemme} Suppose that $\Sigma$ is a compact self-shrinker but $\{s ; H(\gamma(s))<0 \} \neq \varnothing$. Then, by changing $u$ to $-u$ if necessary, there would be $s_1 < t < s_2 \in]0,1[$ such that the following properties are satisfied:

\begin{enumerate}

\item $\gamma{[s_1,s_2]}$ is the graph of a function $f : [x(s_1),x(s_2)] \to \mathbb{R}$ over the $x$-axis,

\item $\langle v , \nu(\gamma(s)) \rangle < 0 $ for all $s \in [s_1,s_2]$,

\item $f$ attains its maximum at $x(t) \in ]x(s_1),x(s_2][$ and $f(x(s_i))<f(x(t))$ for $i=1,2$.   

\end{enumerate}

\end{lemme}	

\begin{proof}(of the lemma)
With a small abuse of notation, we write $\nu(s)$ for $\nu(\gamma(s))$.
Let $a$ be such that $H(\gamma(a))<0$. By changing $u$ to $-u$ if necessary, we can suppose by continuity that $x(a) < 0$. We can also suppose that $\nu(a) \neq \pm u$. Define the functions $\rho_1(a) = \sup\{ s< a ; \nu(s) = \pm u \}$ and $\rho_2(a)= \inf\{ s> a ; \nu(s) = \pm u \}$. Let's distinguish two cases, depending on the sign of $\langle v , \nu(a) \rangle $.

If $\langle v , \nu(a) \rangle > 0$ then, the fact that $x(a)<0$ and the usual maximum principle imply that $\nu(\rho_2(a)) = u$ (if not, consider a plane touching locally $\Sigma$ between $\gamma(a)$ and $ \gamma(\rho_2(a))$. Once again by this argument, there exists $a'$ greater than $\rho_2(a) $ such that $x(\rho_1(a'))<x(a')$ and $\nu(\rho_1(a'))=u$. Now either $\nu(\rho_2(a'))=-u$ or $\nu(\rho_2(a'))=u$. In the first case, let 
$$s_1=\rho_1(a'), s_2=\rho_2(a')$$ 
and take $t \in ]s_1,s_2[ $ such that $r(t)$ is maximal when $t \in ]s_1,s_2[$: the properties of the lemma are indeed satisfied. In the second case (the two normal vectors have the same direction, see figure \ref{dessinlemme}): consider the rotationally symmetric hypersurface $\mathcal{B}$ generated by $\gamma_{|[\rho_1(a'),\rho_2(a')]} : [\rho_1(a'),\rho_2(a')] \to \mathbb{R} \times \mathbb{R}^+$. Define the function $$\tau :  z \in \mathcal{B} \to \langle \nu_z , u \rangle.$$ On the boundary $\partial \mathcal{B}$, $\tau$ is equal to its greater possible value $1$, so this function attains its minimum inside $\mathcal{B}$. This minimum is strictly less than $1$ as $s \in [\rho_1(a'),\rho_2(a')] \to x(\gamma(s))$ is not constant. Besides, if $A$ denotes the second fundamental form, $|A|^2 \neq 0$ as soon as $\tau \neq 1$ or $-1$ because of the rotational symmetry. Using the self-shrinker equation \ref{self} and the Codazzi equations, we compute in local charts:

\begin{align*}
\nabla \tau = A.u^T
\end{align*}
then
\begin{align*}
\Delta \tau & = g^{ij} \nabla_i\nabla_j \langle \nu , u \rangle \\
& = g^{ij} \nabla_i \langle \nabla_j \nu , u \rangle \\
& = g^{ij} \nabla_i \langle h_{jl} g^{lm} \frac{\partial X}{\partial x^m} , u \rangle \\
& = \langle g^{ij} \nabla_i h_{jl} . g^{lm} \frac{\partial X}{\partial x_m} , u \rangle + \langle h_{jl}.g^{lm} g^{ij} \nabla_i \frac{\partial X }{\partial x^m} , u \rangle \\
& = \langle \nabla H , u \rangle - |A|^2 \tau \\
& = \langle X , A.u^T \rangle - |A|^2 \tau.
\end{align*}

\begin{figure} 
\includegraphics[trim = 0mm 80mm 0mm 0mm, clip, width=8cm]{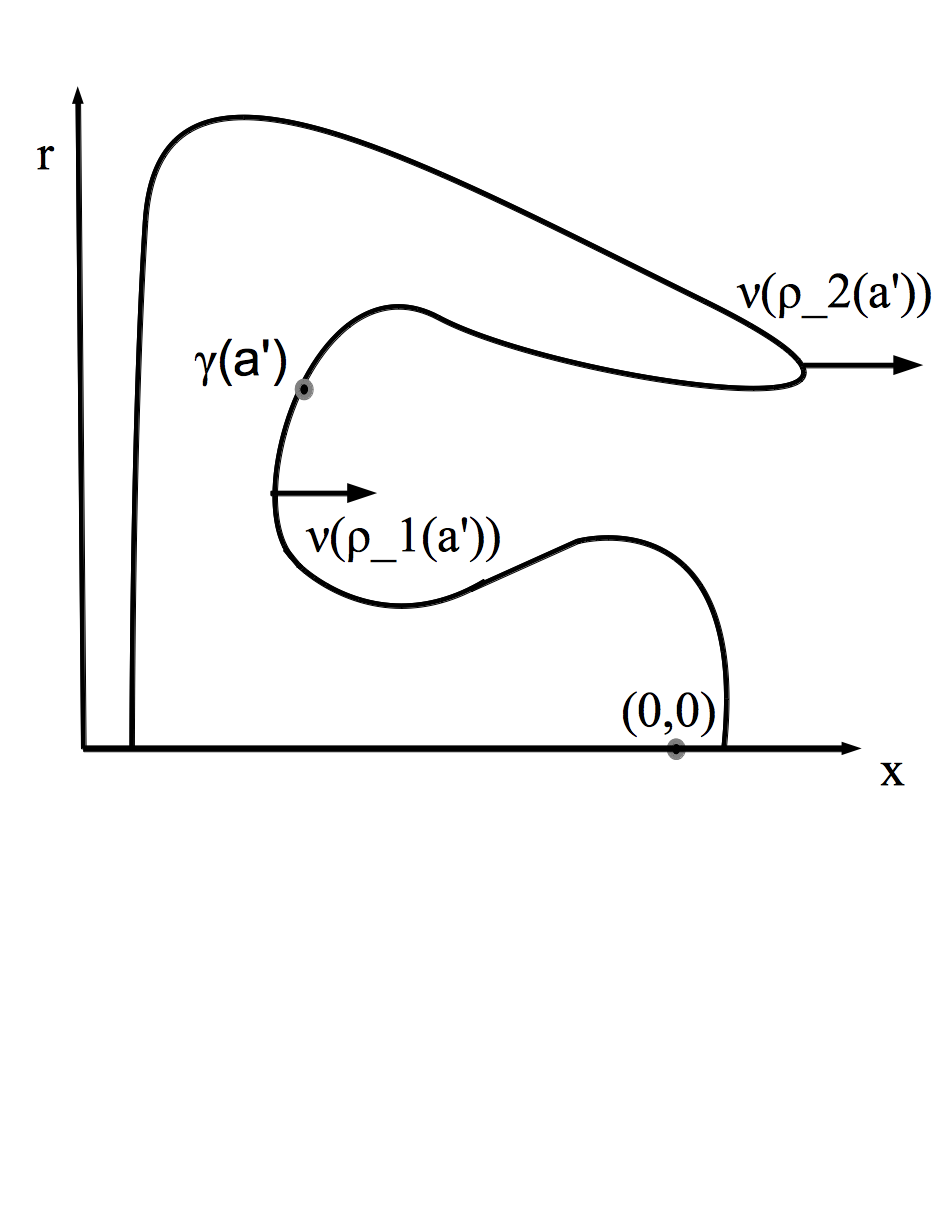}
\caption{Case where $\langle v , \nu(a) \rangle > 0$ and the two considered normal vectors have the same direction.}
\label{dessinlemme}
\end{figure}

The function $\tau$ is thus strictly negative at a point of minimum, which makes it possible to find $s_1<t<s_2$ as in the lemma (recall that $x(\rho_1(a'))<x(a')$).

If $\langle v , \nu(a) \rangle < 0$, then either $\langle u , \nu(a) \rangle >0$, or $\langle u , \nu(a) \rangle < 0$. Consider the first case, the second one being similar. If $\nu(\rho_2(a)) = -u$ then the lemma is verified. Suppose that $\nu(\rho_2(a)) = u$. If $\nu(\rho_1(a)) = u$, by arguing as in the preceding case where the two normal vectors considered have the same direction. Finally, there is the case $\nu(\rho_2(a)) = u$ and $\nu(\rho_1(a)) = -u$: a fortiori, $x(\rho_1(a)) < x(\rho_2(a))$. Remind that $x(0) > x(1)$. The lemma is then verified by considering a neighborhood of the point where 
 $$\min \big \{ \maxÊ\{ r(s) ; s \in [0,\rho_1(a)] \} ; \max \{ r(s) ; s \in [\rho_2(a),1] \} \big \}$$
is achieved.

\end{proof}

To conclude the proof of the theorem, let $s_1$, $s_2$, $a$ be as in the lemma. Of course we can suppose that $r(s_1) = r(s_2)$ and that there is no $s$ between $s_1$ and $s_2$ such that $r(s) = r(s_1)$ and $x(s_1) < x(s) < x(s_2)$. If $\min\{ \maxÊ\{r(s) ; s\in [0,s_1]\} ; \max\{r(s) ; s \in [s_2,1]\}\}$ is attained at $s \in [0,s_1[$ then define $\mathcal{S}^-$ as the hypersurface generated by $\gamma_{| [0,s_1[}$, if not define $\mathcal{S}^-$ as generated by $\gamma_{ | ]s_2,1]}$. Let $\mathcal{S}^+$ be the complement of the closure of $\mathcal{S}^-$ in $\Sigma$. Apply \ref{comparaison} with $e_{n+1} \approx u$ in the first case, $e_{n+1} \approx -u$ in the second case where the sign $\approx \pm u$ means that $e$ is chosen very close to $\pm u$ (essentially $\Sigma$ looks like the left of Figure \ref{configurations}).

\end{proof}

\subsection{Asymptotic behavior and cones}

We now apply Proposition \ref{PMsimple} to get a simple proof of a result of Ilmanen (see \cite{Ilmanen} Lecture $2$, $B$ , remark on p.$8$).

\begin{definition}
Let $K$ be a compact subset of $\mathbb{S}^n$. The set $$\{ rx | r>0 \quad x\in K\} \cup \{0\} \subset \mathbb{R}^{n+1}$$
is called the cone generated by $K$.
\end{definition}

As usual, $S(0,r)$ (resp. $B(0,r)$) denotes the sphere (resp. the ball) of radius $r$ centered at the origin. 

\begin{theo}  [\cite{Ilmanen}]   \label{cone}
Let  $\Sigma^n \subset \mathbb{R}^{n+1}$ be a complete properly immersed self-shrinker. Then there exists a cone $\mathcal{C} \subset \mathbb{R}^{n+1}$ generated by a compact set $K \subset \mathbb{S}^n$ such that:
$$\lambda \Sigma \to \mathcal{C} \text{ as } {\lambda \to 0^+},$$
locally for the Hausdorff metric.

Moreover in the case $n=2$, if the number of connected components of $S(0,r) \cap \Sigma$ is bounded when $r \to \infty$, then the compact $K$ is of $2$-dimensional Lebesgue measure $0$. 
\end{theo}

\begin{proof}

Denote by $d_H$ the Hausdorff distance for non-empty compact sets in $\mathbb{R}^{n+1}$ and $\Sigma_\lambda = \lambda \Sigma $ for $\lambda>0$. Let  $r>0$ fixed. Define
$$K_r^\lambda = \Sigma_\lambda \cap S(0,r).$$ 
It is sufficient to show that $K_r^\lambda$ converges in the Hausdorff metric to a compact set $K_r\subset S(0,r)$ when $\lambda$ goes to $0$.

Let's introduce some other notations:
if $0<\lambda<r/\sqrt{2n}$ we write $$L_r^\lambda = \{ p \in S(0,r)  ;  d(p,K_r^\lambda) \leq \sqrt{2n}\lambda \}$$
where $d$ is the distance between a point and a closed set,
 $$M^\lambda _r = \cap_{\lambda \leq \mu \leq \frac{r}{\sqrt{2n}}} L^\mu _r,$$
 $$Y_r = \Sigma \setminus B(0,r)$$ 
and $$Z_r^\lambda = \{ p\in \mathbb{R}^{n+1} ; ||p|| \geq r \text{  and  } r\frac{p}{||p||} \in L_r^\lambda \}.$$

Firstly, let's show a kind of monotonicity relation:
\begin{equation} \label{monotonie}
\forall \lambda, \mu \in]0,\frac{r}{\sqrt{2n}}[ \quad \mu \leq \lambda \Rightarrow K^\mu_r \subset L^\lambda_r.
\end{equation}

Note that if $\lambda = \frac{r}{\sqrt{2n}}$, then $K^\lambda_r \neq \varnothing$. Indeed, by homogeneity, this is tantamount to saying that $K^1_{\sqrt{2n}} \neq \varnothing$, which is true by the Maximum Principle. Likewise, to prove \ref{monotonie}, we need to show
$$\forall r>\sqrt{2n} \quad Y_r \subset Z^1_r .$$
To prove this, suppose that there is $a\in S(0,r) \setminus L^1_r $ and $b \in Y_r$ such that $r\frac{b}{||b||} = a$. Choose $e_{n+1}$ "very close" to $-\frac{a}{||a||}$, in a sense precised below. Consider then the two following hypersurfaces $\mathcal{S}_1$ and $\mathcal{S}_2$: $\mathcal{S}_1$ is the spherical self-shrinker of radius $\sqrt{2n}$ centered at the origin and $\mathcal{S}_2 = Y_r$. Now, for $e_{n+1}$ well chosen near $-\frac{a}{||a||}$, the function $h : \mathcal{S}_1 \to \mathbb{R} \cup \{ \infty \}$ is not  $\infty$ everywhere and attains its minimum on points not on the boundary of $\mathcal{S}_2$ because $d(a,K^1_r ) > \sqrt{2n}$. We conclude with Proposition \ref{PMsimple}. In particular, $M^\lambda_r$ is not empty if $K^\lambda_r$ is not.

Suppose now that $\Sigma$ is unbounded, otherwise there is nothing to prove. This means that $K^\lambda_r \neq \varnothing$ for $ \lambda\leq \frac{r}{\sqrt{2n}}$. Define $$K_r = \cap_{\lambda\leq \frac{r}{\sqrt{2n}}} L^\lambda_r.$$
The $M^\lambda_r$ being non empty for $\lambda\leq \frac{r}{\sqrt{2n}}$ and included one in the other, $K_r$ is a non empty compact set. Let's show that $K^\lambda_r$ converges to $K_r$. As $K^\lambda_r \subset M^\lambda_r$,
$$\sup_{x \in K^\lambda_r} d(x,M^\lambda_r) = 0.$$
Then by the definition of $M^\lambda_r$, if $y \in M^\lambda_r$, there exists $s \in K^\lambda_r$ such that $d(x,y)\leq \sqrt{2n}\lambda$, so $$\sup_{y \in M^\lambda_r} d(y,K^\lambda_r) \leq  \sqrt{2n}\lambda.$$
Now, $K_r$ is the intersection of the $M^\lambda_r$ which constitute a decreasing nested sequence of non empty compact sets so $M^\lambda_r$ converge to $K_r$. Finally,
\begin{align*}
d_H(K^\lambda_r,K_r) & \leq d_H(K^\lambda_r , M^\lambda_r) + d_H(M^\lambda_r,K_r) \\
& \to_{\lambda \to 0} 0,
\end{align*}
which is the desired convergence.

In what follows, suppose that $n=2$, that $\Sigma$ is properly immersed and that the number of connected components of $S(0,r) \cap \Sigma$ is bounded when $r \to \infty$. The $d$-dimensional Lebesgue measure of a submanifold of $\Sigma$ its $d$-volume, denoted by $\mathrm{Vol}_d(M)$. We have

\begin{equation} \label{aire}
\mathrm{Vol}_2(L^\lambda_r) \to 0  \text{  as  } \lambda \to 0^+.
\end{equation}	

To show this, we use the Euclidean volume growth for properly immersed self-shrinkers (\cite{DingXin}). Because $K_r$ is the intersection of the $L^\lambda_r$, it is sufficient to prove that there is a sequence $(\lambda_k)_{k \in \mathbb{N}}$ converging to $0$ with $\mathrm{Vol}_2(L^{\lambda_k}_{r}) \to 0$. The Sard Lemma implies that for almost all $\lambda>0$, $K^\lambda_r$ is a $1$-dimensional submanifold of $S(0,r)$. Let $\Lambda$ be the set of those $\lambda$. Suppose that there is an $\epsilon>0$ and a $\lambda_0>0$ such that
\begin{equation} \label{absurde}
\forall \lambda<\lambda_0, \mathrm{Vol}_2(L^{\lambda}_{r}) > \epsilon,
\end{equation}	
we want to find a contradiction.

\begin{lemme}
Take $r>0$, $0<\eta <1$. Suppose that  $\gamma \subset S(0,r) \subset \mathbb{R}^3$ is an immersed closed curve of length $l$. Then there exist $C_1$, $C_2$ two constants independent of $\gamma$, $r$ and $\eta$ such that  
$$\mathrm{Vol}_2( \{x \in S(0,r) ; d(x,\gamma) \leq \eta\}) \leq \eta.(C_1.l+C_2.\eta).$$
\end{lemme}

\begin{proof}(of the lemma)

Let $k$ be the maximal number of discs of radius $\eta$ (in $S(0,r)$) such that they are disjoint and centered on a point of $\gamma$. As $\gamma$ is connected, $$(k-1) \leq \frac{l}{2 \eta}.$$
Let $\mathcal{F}$ be a family of such discs, with $k$ elements. It is not empty (i.e. $k\neq0$) and the distance between $\bigcup_{D \in \mathcal{F}} D$ and a point of $\gamma$ which is not in $\bigcup_{D \in \mathcal{F}} D$ is less than $\eta$: otherwise one could add to $\mathcal{F}$ the disc centered on this point, which would contradict the maximality of $\mathcal{F}$. Consequently,
$$\mathrm{Vol}_2(\Gamma) \leq \tilde{C} k. \pi (3\eta)^2 \leq \tilde{C} (\pi (3\eta)^2 + \frac{l}{2 \eta}\pi (3\eta)^2) \leq \tilde{C}(\pi (3\eta)^2 + \frac{9}{2} \pi l.\eta)$$
where $\tilde{C}$ depends on $r$ and the lemma is proved.
\end{proof}

This lemma and the assumption (\ref{absurde}) imply that $$\mathrm{Vol}_{1}(K^\lambda_r) \to \infty$$ when $\lambda \in \Lambda$ goes to $0$. Equivalently, one can write that $$\mathrm{Vol}_{1}(K^1_r) = \sigma(r)r,$$ where $\sigma$ is a stricly positive function defined for almost all $r$ and converging to $\infty$. But, if $N$ is the norm function of $\mathbb{R}^{n+1}$, the co-area formula gives:

\begin{align*}
\mathrm{Vol}_2(\Sigma \cap \{\sqrt{2.2} \leq N \leq R\}) & = \int_{\Sigma \cap \{2 \leq N \leq R\}} d\mathrm{Vol}_2 \\
& \geq  \int_{\Sigma \cap \{{2} \leq N \leq R\}} ||\nabla N|| d\mathrm{Vol}_2 \\
& = \int_{{2}}^R \int_{K^1_r} d\mathrm{Vol}_{1} dr \\
& = \int_{{2}}^R \sigma(r)r dr \\
& \geq \frac{1}{2}(R^2-R_0^2) \min_{[R_0,R]}{\sigma}
\end{align*}
where $R>R_0$ are two real numbers greater than $2$, and $ \{2 \leq N \leq R\}$ denotes the set $B(0,R)\setminus B(0,{2})$. As $\sigma \to \infty$, this computation contradicts the quadratic volume growth of $\Sigma$, so in fact $\mathrm{Vol}_2(K_r)=0$. 

\end{proof}

Next, we give two propositions which are based on a refined form of the argument used previously to show the monotonicity relation, which was the key argument for proving Theorem \ref{cone}.

For a point $q\in \mathbb{R}^{n+1}$ and a vector $v \in \mathbb{R}^{n+1}$, we denote by $L(q,v)$ the half-line beginning at $q$ and whose direction is given by $v$. Recall also the notation $K^1_r = \Sigma \cap S(0,r)$ where $\Sigma$ is a complete properly immersed self-shrinker.

\begin{prop} \label{piege}
Let $r>0$ and $p_0$ be a point in $S(0,r)$ (also considered as a vector in $\mathbb{R}^{n+1}$). Define 
$$U_r(p_0) = \{q \in \mathbb{R}^{n+1}\setminus B(0,r) ; L(q,-p_0) \cap S(0,\sqrt{2n}) \neq \varnothing \},$$
 $$V_r = \bigcup_{p\in S(0,r) \text{  and  } U_r(p) \cap K^1_r=\varnothing} U_r(p) .$$
Consider the sets
$$Y_r = \Sigma \backslash B(0,r), $$
 $$X_r = \mathbb{R}^{n+1} \setminus (B(0,r) \cup V_r).$$
Then 
$$\forall r>\sqrt{2n} \quad Y_r \subset X_r.$$
\end{prop}

\begin{proof}
Suppose the contrary: there is $p \in V_r \cap Y_r$. Apply Proposition \ref{PMsimple} to the following hypersurfaces: $\mathcal{S}_1$ is the self-shrinking sphere, and $\mathcal{S}_2$ is $Y_r$. Let $\Pi$ be the orthogonal projection on the vectorial hyperplane orthogonal to $p$, defined on the half-space $E = \{x\in \mathbb{R}^{n+1} ; \langle x , p\rangle>0\}$. Then, by the definition of $V_r$, $\Pi(B(0,\sqrt{2n}) \cap \Pi(E \cap K^1_r) \neq \varnothing$. Thanks to the Sard Lemma, Proposition \ref{PMsimple} is then applied with $e_{n+1}$ near $-p/||p||$ (see figure \ref{piegeimage}).

\end{proof}

 \begin{figure} 
\includegraphics[scale=0.4]{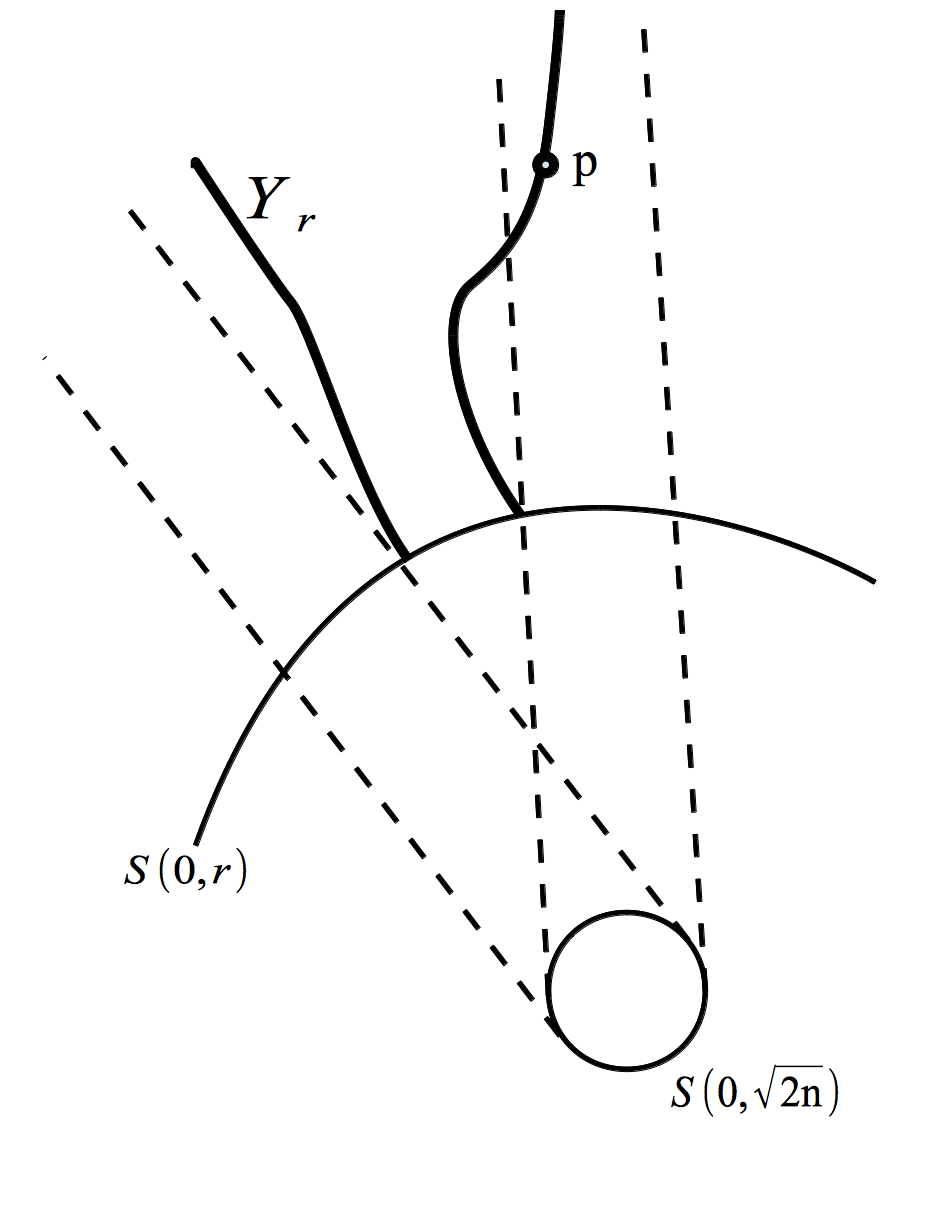}
\caption{In this figure, $\Sigma$ cannot be a self-shrinker because the point $p$ is out of $X_r$.}
\label{piegeimage}
\end{figure}

The following proposition illustrates the fact that some bound on the curvature of the trace $K^1_r=\Sigma\cap S(0,r)$ gives a bound on the mean curvature of $\Sigma$ in the case where $n=2$. 
\begin{prop} \label{boundH}
Suppose $n=2$ and let $\Sigma_0$ be an end of $\Sigma \subset \mathbb{R}^{3}$, i.e. a connected component of $\Sigma \setminus B(0,r_0)$ for a $r_0>0$. Let $\epsilon>0$. If $\Sigma_0$ intersects $S(0,r)$ transversally for all $r>r_0$ (in particular $\Sigma_0 \cap S(0,r)$ a union of closed simple curves) and if the curvature of $\Sigma_0 \cap S(0,r)$ is bounded by $\frac{1}{2}-\epsilon$ then the mean curvature of $\Sigma_0$ is bounded by ${2}$ if $r_0$ is chosen large enough.
\end{prop}

\begin{proof}
Define $I_r = \Sigma_0 \cap S(0,r)$. For $p \in S(0,r)$, denote by $b_r(p)$ the intersection $B(q,{2}) \cap S(0,r)$ where $q$ is parallel to $p$ with the same direction and maximizes the $2$-volume of this intersection. As $r_0$ increases, $S(0,r)$ ($r>r_0$) becomes flat and the curve $\partial b_r(p)$ has curvature larger than ${(1-\epsilon)}/{2}$.

Consequently, if $x_0 \in I_r$, there exists $p_1,p_2 \in S(0,r)$ such that the $b_r(p_i)$ are tangent to $I_r$ at $x_0$ (one on each side of $x_0$) and $I_r$ is outside $b_r(p_1) \cup b_r(p_2)$ near $x_0$. To prove the proposition, it is sufficient to show that a neighborhood of $x_0$ in $\Sigma \backslash B(0,r)$ is outside $U_r(p_1) \cup U_r(p_2)$ (see notation in Proposition \ref{piege}). Indeed, the bound on the mean curvature will then follow from the self-shrinker equation (\ref{self}).
Suppose the contrary: for instance there are points $y_k \in \Sigma \backslash B(0,r)$ converging to $x_0$ with $y_k \in U_r(p_1)$. Consider $C$ the connected component of $(\Sigma \backslash B(0,r))\cap U_r(p_1)$ containing $y_k$, well defined for $k$ large. Note that because of the bound on the curvature, the boundary $\partial C$ only intersects $b_r(p_1)$ at $x_0$. This can be seen as follows: suppose that $\partial C$ intersects $b_{r}(p_1)$ at another point so that for $r'$ slightly bigger than $r$, $C \cap S(0,r') \cap U_r'(p_1)$ has more than one component. As $r'$ increases from $r$, by connectivity of $C$, two of these components of $C \cap S(0,r') \cap U_r'(p_1)$ have to meet smoothly for an $r'>r$. But this can only happen on the boundary $\partial b_{r'}(p)$, which would contradict the bound on the curvature of $\Sigma_0 \cap S(0,r)$. Then we apply once again Lemma \ref{PMsimple} to $S(0,{2})$ and $C$, with the Sard Lemma to guarantee that the function $h$ achieves a minimum.
\end{proof}

 \vspace{1cm}

\section{A linear bound for the second fundamental form of some self-shrinkring surfaces} \label{principale}

Let's define the curvature concentration:
\begin{definition}  \label{concentration}
Let $\kappa>0$. Let $\mathcal{S}$ be a surface in $\mathbb{R}^3$. The curvature concentration of $\mathcal{S}$ is bounded by $\kappa$ if 
$$\forall x \in \mathcal{S} \backslash B(0,1),\quad  \int_{\mathcal{S} \cap B(x,\rho(x))}  |A|^2 < \kappa,$$
where
$$\rho(x) =\frac{1}{2||x||}.$$
\end{definition}

Requiring that a surface has bounded curvature concentration seems quite restrictive, but it is in fact natural in the context of Brakke limit flows:

\begin{prop} \label{naturality}
Let $\{M\}_{t \in [0,T)}$ be a family of embedded surfaces flowing smoothly in $\mathbb{R}^3$ and beginning at a closed surface $M_0$. If $\mathcal{S}$ is a self-shrinker produced by a weak blowup at $T$ then there exists a constant $\kappa(M_0)>0$ depending only on $M_0$, such that the curvature concentration of $\mathcal{S}$ is bounded by $\kappa(M_0)$.

\end{prop}

\begin{proof}
Let's recall the definition of a weak blowup: consider the family $\{M\}_{t \in [0,T)}$ as above. Let $\lambda_i >0$ be a sequence converging to $0$. Rescale the flow parabolically about $(y,T)$ for some $y$ by defining
$$M^i_t = \lambda_i^{-1} . (M_{T+\lambda_i^2t}-y), \quad t \in [-T/\lambda_i^2,0).$$
Ilmanen \cite{Ilmanen2} and White \cite{White} proved that by taking a subsequence if necessary, $\{M^i_t\}_t $ converges to a limiting Brakke self-shrinking flow $\{\sqrt{-t} \mathcal{S}\}_{t \in (-\infty, 0)}$ in the following sense: 

\begin{enumerate}
\item for all $t<0$, $M^i_t \to \sqrt{-t} \mathcal{S}$ in the sense of Radon measures,
\item for a.e. $t<0$, there is a subsequence $\{i_k\}$ depending on $t$ such that $M^i_t \to \sqrt{-t} \mathcal{S}$ as varifolds.
\end{enumerate}
Moreover, Ilmanen showed that in dimension $3$, the limit flow is smooth. This procedure is called a weak blowup. Then Ilmanen gives the following integral curvature estimate (Theorem $4$ in \cite{Ilmanen2}): for every $B(x,r)\times [t-r^2,t) \subset \mathbb{R}^3\times[0,T)$,
\begin{equation} \label{curvature estimate}
r^{-2} \int_{t-r^2}^t\int_{M_t \cap B(x,r)} |A|^2 \leq C(M_0).
\end{equation}
By lower semicontinuity of the integral of the squared norm of the generalized second fundamental form (see \cite{Hutchinson}, Theorem $5.3.2$), it yields the first part of the proposition. Indeed, without loss of generality suppose that $y=0$, and consider $z \in \mathcal{S} \subset{R}^3$ whose norm is bigger than $1$. Applying (\ref{curvature estimate}) to $t = T-\lambda_i^2 ||z||^2$, $x={\lambda_i} ||z||.z$ and $r=\lambda_i$ for all $i$, we obtain a subsequence $i_k$ and a time $t' \in [-||z||^2-1,-||z||^2]$ such that 
\begin{enumerate}
\item $M^{i_k}_{t'} \cap B(||z||.z,1)$ converge to $\sqrt{-t'}\mathcal{S} \cap B(||z||.z,1)$ as varifolds, 
\item $B(\sqrt{-t'}z,\sqrt{-t'}/(2||z||)) \subset B(||z||.z,1)$,
\item $\int_{M^{i_k}_{t'} \cap B(||z||.z,1)} |A|^2$ is bounded by a constant depending only on $M_0$. 
\end{enumerate}
Consequently, after rescaling, we get that $\int_{\mathcal{S} \cap B(z,\rho(z))} |A|^2$ is bounded by a constant depending only on $M_0$, where $\rho(z) =\frac{1}{2||z||}$.

\end{proof}

The main theorem of this section gives a linear bound for $|A|$:

\begin{theo} \label{MAIN}
If $\Sigma \subset \mathbb{R}^{3}$ is a complete properly embedded self-shrinker of finite genus $g$ such that the curvature concentration is bounded by $\kappa$, then
 $$\exists C=C(\Sigma), \forall x \in \Sigma, |A(x)| \leq C(1+ |x|).$$ 
 \end{theo}

\begin{remarque}
Self-shrinkers can be viewed as minimal surfaces under a conformal change of metric but, as noted in \cite{CM2}, this kind of result does not follow for instance from Choi-Schoen's compactness theorem \cite{ChoiSchoen} mainly because the new metric can not even be extended to a complete metric.
\end{remarque}

\begin{proof}
Let $K=(H^2-|A|^2)/2$ denote the Gauss curvature. Suppose that the conclusion is not verified along a sequence of points $p_k \in \Sigma$. Define $\mu_k = {|A(p_k)|}$. Then, because of the self-shrinker equation (\ref{self}):

\begin{equation} \label{k}
\forall \alpha>0 \quad \max_{y \in B(p_k,\alpha/\mu_k)} H(y)/\mu_k \to 0.
\end{equation}
Moreover, $-K(p_k)/|A(p_k)|^2 \to 1/2$.
By modifying the $p_k$ if necessary, $\mu_k . (\Sigma - p_k)$ is a sequence of surfaces whose second fundamental form is locally uniformly bounded and whose mean curvature goes to zero. Indeed, we have

\begin{lemme} \label{technical}
Under the hypotheses of the theorem, there is a sequence $p_k $ such that
 $$ \forall \alpha>0 , \quad |A(p_k)|/(1+||p_k)|) \to \infty  \text{ and }$$
 $$\max \{|A(q)|/\mu_k ; q \in \Sigma \cap B(p_k, \alpha/\mu_k)\}  \leq 2 \text{  for large $k$.   } $$

\end{lemme}  
\begin{proof}(of the lemma)
The proof goes by induction. If $p \in \Sigma$, $|A(p)|>0$ and $\alpha>0$, define $$m(p,\alpha)=\max \{|A(q)|/|A(p)| ; q \in \Sigma \cap B(p, \alpha/|A(p)|)\}.$$ Let $n \in \mathbb{N}$. Denote by $\mathcal{P}(n)$ the following assertion:

There is a sequence $p_k \in \Sigma$ such that
$$|A(p_k)|/(1+||p_k||) \to \infty \text{  and  } \forall k\geq n , m(p_k,n)\leq 2.$$

Because $\mathcal{P}(0)$ is trivially true, the lemma will ensue by a diagonal extraction argument from the following claim: if $\mathcal{P}(n)$ is verified then by modifying $p_k$ if necessary (for $k\geq n+1$), we have $\mathcal{P}(n+1)$.

Let's check this claim: suppose that $\mathcal{P}(n)$ is true for the sequence $(x_k)$. Fix a $k$ bigger than or equal to $n+1$. One can suppose that $(n+1)/|A(x_k)|\leq1$. If $m(x_k,n+1)>2$ then there is a point $x_k^1 \in \Sigma  \cap B(x_k ,1/2)$ with $$|A(x_k^1)|/|A(x_k)| >2.$$
Likewise, if $m(x_k^1,n+1) > 2$, we can find $x_k^2 \in \SigmaÊ\cap B( x_k^1, 1/2^2 )$ such that
$$ |A(x_k^2)| / |A(x_k^1)| >2.$$ 

This construction goes on as long as $m({x_k^l},n+1) >2$. In fact, it has to stop because $d(x_k^l,x_k^{l+1}) \leq 1/2^{l+1}$ and $|A({x_k^l})| \geq 2^l |A(x_k)|$. Let's call $p_k$ the last ${x_k^l}$ constructed and of course define $x_k=p_k$ for $k\leq n$. The sequence $(p_k)$ verifies $\mathcal{P}(n+1)$. 

\end{proof}

Define now $\mathbf{S}_k =\mu_k . (\Sigma - p_k)$. The previous lemma shows that $\mathbf{S}_k$ is a sequence of surfaces such that at the origin, the Gauss curvature is $-1/2$ and the second fundamental form is locally uniformly bounded. Besides, (\ref{k}) implies that for all $a>0$ and all $x_k \in \Sigma \cap B(p_k,a/\mu_k)$, the quantity $H(x_k)/\mu_k$ goes to $ 0$. In other words, the mean curvature of $\mathbf{S}_k$ converges locally uniformly to $0$. The local uniform bound on $|A|$ means that in any ball of $\mathbb{R}^3$ small enough (the radius depends only on this bound), for $k$ large, $ \mathbf{S}_k$ is the graph of a function with bounded gradient and Hessian. Consequently, these functions satisfy uniformly elliptic equations with uniformly controlled coefficients. Thus by Schauder estimates and standard elliptic theory, for all $r>0$, the intrinsic balls $B^{\mathbf{S}_k}(0,r)$ converge subsequently in the $C^m$ topology to a smooth embedded surface with boundary called $\mathbf{T}_r$, which is in fact minimal. Define $\mathbf{S}$ to be the union of the $\mathbf{T}_r$ for $r>0$. It's an embedded complete minimal surface. Here the surfaces considered are all oriented embedded, the integral of $|A|^2$ is bounded uniformly and $\mathbf{S}$ is non flat. Consequently the convergence has in fact multiplicity one, i.e. $B^{\mathbf{S}_k}(0,r)$ converge smoothly with multiplicity one to $B^\mathbf{S}(0,r)= \mathbf{T}_r$. Besides, note that

\begin{equation}  \label{courbure totale}
\forall r>0 \quad   \int_{\mathbf{S} \cap B(0,r)} |A|^2 = 2\int_{\mathbf{S} \cap B(0,r)} |K| \leq \limsup_{k \to \infty} \int_{\mathbf{S}_k \cap B(0,r)} |A|^2 \leq \kappa.
\end{equation}

Now, let's cite two results which will imply that $\mathbf{S}$ is necessarily a catenoid.

The first well-known theorem of Osserman describes minimal surfaces with finite total curvature.
\begin{theo} \cite{Osserman}
Let $M \subset \mathbb{R}^3$ be a complete oriented minimal surface with Gauss curvature $K$ such that $\int_M |K| < \infty$. Then there exists a closed Riemann surface  $\tilde{M}$, a finite set of points $\{p_1, ... , p_k\} \in \tilde{M}$ and a conformal diffeomorphism between $M$ and $\tilde{M}\setminus \{p_1, ... , p_k\}$. Moreover, the Gauss map extends meromorphically across the punctures.
\end{theo}

The following theorem, proved by L\'{o}pez and Ros, then Meeks and Rosenberg, classifies the properly embedded minimal surfaces with finite topology and genus $0$:

\begin{theo} [\cite{LopezRos}, \cite{MeeksRosenberg}]
The only properly embedded minimal surfaces of $\mathbb{R}^3$, with finite topology and genus $0$ are planes, helicoids and catenoids. 
\end{theo}

By invariance under dilatations of integrals like $ \int_{\mathbf{S}_k \cap B(0,r)} |A|^2$, the first of the theorems above, the assumption on the curvature concentration and (\ref{courbure totale}) show that $\mathbf{S}$ has finite topology. Note that moreover, the normal vectors are well defined at each end, which implies that $\mathbf{S}$ is proper. Because the norm of the second fundamental form of $\Sigma$ at $p_k$ goes to infinity, $||p_k|| \to \infty$. Hence, $\mathbf{S}$ has genus $0$ (the convergence has multiplicity one). The second theorem, plus the fact that the Gauss curvature of $\mathbf{S}$ at the origin is $-1/2$, imply that $\mathbf{S}$ is necessarily a catenoid. Note that, as the multiplicity of the convergence is one, for $k$ large $\Sigma_k$ really looks like a small catenoid near $p_k$.

The end of the proof consists in finding a contradiction with the fact that $\mathbf{S}$ is a catenoid. Note that a posteriori, we could have chosen $p_k$ to have the origin on the closed simple geodesic of the catenoid $\mathbf{S}$, and also to make $p_k/||p_k||$ converge to a vector, say $v$. Let $(e'_1,e'_2,e'_3)$ be a basis of $\mathbb{R}^3$ such that $e'_3$ gives the rotation axis of $\mathbf{S}$. 

Let $c \subset \mathbf{S}$ be the closed simple geodesic which encircles the neck of $\mathbf{S}$. Let $(c_k)$ be a sequence of curves in $\Sigma$ corresponding to embedded curves in $\mathbf{S}_k$ converging smoothly to $ c \subset \mathbf{S}$. Let $\Sigma^+_k$ and $\Sigma^-_k$ be the two open connected components separated by $c_k$, $\Sigma^+_k$ being the component whose rescaling converge to $\{x \in \mathbf{S} ; \langle x , e'_3 \rangle>0\}$ (we can suppose that there are two distinct components because the genus of $\Sigma$ is finite and $||p_k|| \to \infty$). Note that seen from afar, a catenoid is like two superposed planes. So for $k$ large, $\Sigma^+_k$ and $\Sigma^-_k$ are two surfaces nearly parallel and flat around $p_k$ and glued together along $c_k$.

There are two cases: either $v$ is parallel to $e'_3$ or it isn't. 

Suppose that $v$ is parallel to $e'_3$, say $v=e'_3$. Then by applying Proposition \ref{piege}, we see that, for $k$ large, $\Sigma^+_k$ should be contained in a set looking like a long tube with boundary $c_k$ going to infinity and as narrow as $c_k$ near this little neck. But $\Sigma^+_k$ is more like a plane orthogonal to $e_3' = p_k/||p_k||$ with a small half neck at $p_k$. This is the desired contradiction.

Suppose otherwise that $v$ is not parallel to $e'_3$. The latter can be chosen so that $\langle e'_3 , v\rangle>0$. In this case, by the usual maximum principal, for all $k$ large
$$\Sigma^{\pm}_k \cap S(0,2) \neq \varnothing.$$
This observation plus the fact that the genus of $\Sigma$ is finite would imply that $\Sigma$ is not properly embedded, which is absurd.

\end{proof}

This theorem implies that translating solitons can help to understand regions of $\Sigma$ far from the origin:

\begin{coro} \label{translating}
Let $\Sigma\subset \mathbb{R}^3$ be as in the previous theorem. Consider a sequence $x_k \in \Sigma$ of points whose norm goes to infinity. Define $\mathbf{S}_k = ||x_k||.(\Sigma-x_k)$. Then subsequently, for all $d>0$, $B^{\mathbf{S}_k}(0,d)$ converge smoothly with multiplicity one to a surface with boundary $T_d$ and the union $T = \cup_d T_d$ is a complete embedded translating soliton of the mean curvature flow with genus $0$.
\end{coro}

\begin{proof}
One can suppose that $x_k/||x_k||$ converge to a vector $u$. Theorem \ref{MAIN} implies that the norm of the second fundamental form of $\Sigma$ grows at most linearly. Hence, the second fundamental form of the rescalings $\mathbf{S}_k$ is bounded. Then the convergence of $B^{\mathbf{S}_k}(0,d)$ to a $T_d$ (up to a subsequence) is shown as in the proof of the previous theorem. Moreover, each limit $T_d$ satisfies
$$H = \frac{\langle u, \nu \rangle}{2}$$
which is the equation of translating solitons.
Note that this time the integral of $|A|^2$ is uniformly bounded only on fixed compact sets, that is 
\begin{equation} \label{viande}
\forall s>0, \exists c=c(s), \forall k, \int_{\mathbf{S}_k\cap B(0,s)} |A|^2 \leq c.
\end{equation}
Take a $d'$ small enough (depending on the constant $C(\Sigma)$ in the theorem) so that $B^{\mathbf{S}_k}(0,d')$ are graphs. Then if $B^{\mathbf{S}_k}(0,d')$ converge to a non flat surface, the multiplicity of the convergence of $B^{\mathbf{S}_k}(0,d)$ for all $d>0$ has to be one because of (\ref{viande}). If $B^{\mathbf{S}_k}(0,d')$ converge to a flat surface, then by unique continuation $B^{\mathbf{S}_k}(0,d)$ converge to a flat surface for all $d>0$ (remember that each $T_d$ is a piece of translating soliton, which is a minimal surface after a conformal change of metric). These arguments show that the convergence of $B^{\mathbf{S}_k}(0,d)$ to $T_d$ has to be one for all $d$.
The genus of $T$ is zero because the convergence has multiplicity one and because $||x_k|| \to \infty$.

\end{proof}

For the end of this section, we give a result analogous to Theorem \ref{MAIN}, but for surfaces with $\delta$-separated ends, so that the constant this time won't depend on the surface itself but on some geometric parameters mentioned in the following definitions:

\begin{definition} \label{bounded}
Let $R>0$. A complete properly embedded surface $\mathcal{S} \subset \mathbb{R}^3$ with finite genus $g(\mathcal{S})$ is said to have topology bounded by $R$ if the genus of $B(0,R) \cap \mathcal{S}$ is equal to $g(\mathcal{S})$.
\end{definition}

\begin{definition}
Define $\mathfrak{F}^{\delta, \kappa}_{g,R}$ the family of complete self-shrinking surfaces properly embedded in $\mathbb{R}^3$ with finite topology such that: 
\begin{enumerate}
\item the genus is less than $g$,
\item the topology is bounded by $R$,
\item the curvature concentration is bounded by $\kappa$,
\item the ends are $\delta$-separated.
\end{enumerate}
\end{definition}

The proof of the following theorem is similar to the previous one but the new point is that the assumption on the ends of the self-shrinking surfaces enables us to use Corollary \ref{comparaison}. More precisely, small catenoid-like necks on self-shrinking surfaces with $\delta$-separated ends are ruled out by this corollary.

\begin{theo}  \label{main}
There exists a constant $C=C(\delta, \kappa,g,R)$ depending only on, $\delta$, $\kappa$, $g$ and $R$ such that
$$\forall \mathcal{S} \in \mathfrak{F}^{\delta, \kappa}_{g,R},\quad \forall x \in \mathcal{S}, \quad |A(x)| \leq C(1+||x||).$$
\end{theo}

\begin{proof}

The beginning of the proof is nearly identical to the previous one, except that this time we consider $p_k \in \mathcal{S}_k \in \mathfrak{F}^{\delta, \kappa}_{g,R} $ hypothetically contradicting the theorem, define $\mathbf{S}_k =\mu_k . (\mathcal{S}_k - p_k)$ and use this lemma:

\begin{lemme} \label{technical2}
Under the hypotheses of the theorem, there is a sequence $(p_k)$ with $p_k \in \mathcal{S}_k \in \mathfrak{F}^{\delta,\kappa}_{g,R} $ such that
 $$ \forall \alpha>0 , \quad |A(p_k)|/(1+||p_k)|) \to \infty  \text{ and }$$
 $$\max \{|A(q)|/\mu_k ; q \in \mathcal{S}_k \cap B(p_k, \alpha/\mu_k)\}  \leq 2 \text{  for large $k$.   } $$

\end{lemme}  

As before, we get a limit $\mathbf{S}$. To prove that it is once again a catenoid, the same arguments work if we can prove that the genus of $\mathbf{S}$ is $0$, which was clear when only one surface was considered. We use a theorem of compactness for self-shrinkers by Colding and Minicozzi.

\begin{theo}[\cite{CM2}]
Let $g$ be a integer. The space $\mathcal{F}_g$ of properly embedded self-shrinkers without boundary, with finite topology and genus less than $g$ is compact for the topology of $C^m$ convergence on compacts.
\end{theo}

\begin{remarque}
Colding and Minicozzi initially supposed that the volume growth was quadratic, with is a natural assumption in the context of limit surfaces for mean curvature flow flowing from embedded closed surface \cite{CM1}. This hypothesis is in fact always true as long as the self-shrinker is proper, as shown by Ding and Xin \cite{DingXin}.
\end{remarque}

From this theorem and the fact that $\mathfrak{F}^{\delta, \kappa}_{g,R} \subset \mathcal{F}_g$, we deduce that $||p_k|| \to \infty$. The topology being uniformly bounded by $R$, $\mathbf{S}$ has indeed genus $0$ (remind that the convergence of $B^{\mathbf{S}_k}(0,r)$ to $B^\mathbf{S}(0,r)$ has multiplicity one).

The end of the proof is based on Corollary \ref{comparaison}, as explained previously. A posteriori, we could have chosen $p_k$ to have the origin on the closed simple geodesic of the catenoid $\mathbf{S}$, and also to make $p_k/||p_k||$ converge to a vector, say $v$. Let $(e'_1,e'_2,e'_3)$ be a basis of $\mathbb{R}^3$ such that $e'_3$ gives the rotation axis of $\mathbf{S}$. We denote by $P$ the plane containing $e'_1$ and $e'_2$.

We will distinguish two cases: either the limit $v$ is not in the plane $P$, or $v$ is in $P$ (i.e. $v$ is orthogonal to $e_3'$). 
 
 Case 1.
 
Suppose that $v$ is not in the plane $P$. Let $c \subset \mathbf{S}$ be the closed simple geodesic which encircles the neck of $\mathbf{S}$. Note that because of the choice of $p_k$ in Lemma \ref{technical2} and because $\mathbf{S}$ and $ \tilde{\mathbf{S}} \subset \mathbb{R}^3$ necessarily intersect if $\tilde{\mathbf{S}}$ is a plane or a catenoid, one can suppose that $B(0,k) \cap \mathbf{S}_k$ has only one connected component for all $k$. Let $(c_k)$ be a sequence of curves in $\mathcal{S}_k$ corresponding to embedded curves in $\mathbf{S}_k$ converging smoothly to $ c \subset \mathbf{S}$. We can suppose that $\langle e'_3 , v\rangle>0$. Let $\mathcal{S}^+_k$ and $\mathcal{S}^-_k$ be the two open connected components separated by $c_k$, $\mathcal{S}^+_k$ being the component whose rescaling converge to $\{x \in \mathbf{S} ; \langle x , e'_3 \rangle>0\}$ (we can suppose that there are two disctinct components because the topology of the $\mathbf{S}_k$ is bounded and $||p_k|| \to \infty$). Note that seen from afar, a catenoid is like two superposed planes. So for $k$ large, $\mathcal{S}^+_k$ and $\mathcal{S}^-_k$ are two surfaces nearly parallel and flat around $p_k$ and glued together along $c_k$.
 
Suppose first that $\mathcal{S}^+_k$ intersects the self-shrinking sphere $S(0,2)$: this is in fact always the case if $v \neq e_3'$ by the usual maximum principle, because then for $k$ large there are points in $\mathcal{S}^+_k$ which are closer to the origin than any points on $c_k$. Then, we can use Corollary \ref{comparaison}. Indeed, take $y \in \mathcal{S}^+_k \cap S(0,2)$. By changing the canonical basis if necessary, we can suppose that  $e_3 = (p_k-y)/||p_k-y||$. Let's check for instance the second hypothesis of Corollary \ref{comparaison}. For every point $q$ in a tubular neighborhood of $c_k$, for $k$ large, the half-line $[q,y)$ first touches $\mathcal{S}^-_k$ before arriving at $y$ (recall that $B(0,k) \cap \mathbf{S}_k$ has only one connected component for all $k$.). By the same kind of arguments, the hypothesis of Corollary \ref{comparaison} are all verified. So $\mathcal{S}_k$ is not a self-shrinker, which is absurd.
 
Suppose now that $\mathcal{S}^+_k$ does not intersect the self-shrinking sphere $S(0,2)$: necessarily $v=e'_3$. By applying Proposition \ref{piege}, we see that, for $k$ large, $\mathcal{S}^+_k$ should be contained in a set looking like a long tube with boundary $c_k$ going to infinity and as narrow as $c_k$ near this little neck. But $\mathcal{S}^+_k$ is more like a plane orthogonal to $e_3' = p_k/||p_k||$ with a small half neck at $p_k$. This is the desired contradiction.

Case 2.

Suppose that $v \in P$. Choose a sequence of embedded curve $(c_k) \subset \mathcal{S}_k$ as before, let $\mathcal{S}^+_k$ and $\mathcal{S}^-_k$ be the two open connected components separated by $c_k$, $\mathcal{S}^+_k$ being the component whose rescaling converge to $\{x \in \mathbf{S} ; \langle x , e'_3 \rangle>0\}$. By Theorem \ref{cone}, we know that $\lambda . \mathcal{S}^+_k$ (resp. $\lambda . \mathcal{S}^-_k$) converge (when $\lambda\to 0$) to a cone whose intersection with $S(0,1)$ is called $K^+_k$ (resp. $K^-_k$). By $\delta$-separation of the ends of $\mathcal{S}_k$, $d(K^+_k,K^-_k) \geq \delta$. Define $\Theta_{e,\theta}$ the linear rotation with axis $\mathbb{R}e$ and angle $\theta$. Consider the axis $e_0=v \wedge e'_3$. Because $\mathcal{S}_k$, $d(K^+_k,K^-_k) \geq \delta$, for $\theta$ sufficiently small independent of $k$, there is a $\overline{r}$ depending on $k$ such that if $r>\overline{r}$, 
$$\Theta_{e,\theta}(\mathcal{S}^+_k)  \setminus B(0,r) \cap \mathcal{S}^-_k \setminus B(0,r) = \varnothing.$$
Now, fix an angle $0<\theta_0<\pi$ sufficiently small. With these choices,
$$\forall k \text{ large, }\Theta_{e_0,\theta_0}(\mathcal{S}^+_k)  \cap \mathcal{S}^-_k = \varnothing.$$
Indeed, if it is not true for $k'$, then consider the first $\theta$ between $0$ and $\theta_0$ such that $\Theta_{e,\theta}(\mathcal{S}^+_{k'}) $ touches $\mathcal{S}^-_{k'}$. Then these surfaces are in contact at a point in the ball $B(0,\overline{r})$ and not outside. But that contradicts the usual maximum principle. We can finally conclude, because for $k$ large, $$\inf \{ ||x-y||;x\in \Theta_{e_0,\theta_0}(\mathcal{S}^+_k) , y\in \mathcal{S}^-_k \}$$ will be achieved since $\theta_0$ is independent of $k$: once again the usual maximum principle is violated so $\mathcal{S}_k$ can't be self-shrinkers for $k$ large, which is absurd.

\end{proof}

\bibliographystyle{plain}
\bibliography{bibliographiedu051214}

\begin{thebibliography}{10}

\bibitem{AbreschLanger}
U.~Abresch and J.~Langer.
\newblock The normalized curve shortening flow and homothetic solutions.
\newblock {\em J. Differential Geom.}, 23(2):175--196, 1986.

\bibitem{Angenent}
S.~Angenent.
\newblock Shrinking doughnuts.
\newblock In {\em Nonlinear diffusion equations and their equilibrium states},
  number~3, pages 21--38. Birkha\"user, 1992.

\bibitem{ChoiSchoen}
H.~I. Choi and R.~Schoen.
\newblock The space of minimal embeddings of a surface into a three-dimensional
  manifold of positive {R}icci curvature.
\newblock {\em Invent. Math.}, 81:387--394, 1985.

\bibitem{CM1}
T.~H. Colding and W.~P. {Minicozzi II}.
\newblock Generic mean curvature flow i; generic singularities.
\newblock {\em Ann. of Math.}, 175(2):755--833, 2012.

\bibitem{CM2}
T.~H. Colding and W.~P. {Minicozzi II}.
\newblock Smooth compactness of self-shrinkers.
\newblock {\em Commentarii Mathematici Helvetici}, 87(2):463--475, 2012.

\bibitem{DingXin}
Q.~Ding and Y.~L. Xin.
\newblock Volume growth eigenvalue and compactness for self-shrinkers.
\newblock {\em Asian J. Math.}, 17(3):443--456, 2013.

\bibitem{Drugan}
G.~Drugan.
\newblock Embedded self-shrinkers $\mathbb{S}^2$ with rotational symmetry.
\newblock 1998.

\bibitem{Huisken1}
G.~Huisken.
\newblock Asymptotic behavior for singularities of the mean curvature flow.
\newblock {\em J. Differential Geom.}, 31(1):285--299, 1990.

\bibitem{Huisken2}
G.~Huisken.
\newblock Local and global behaviour of hypersurfaces moving by mean curvature
  flow.
\newblock In {\em Differential Geometry. Part. 1 : partial differential
  equations on manifolds}, volume 54 (1), pages 175--191. Amer. Math. Soc.,
  1993.

\bibitem{Hutchinson}
John~E. {Hutchinson}.
\newblock {Second fundamental form for varifolds and the existence of surfaces
  minimising curvature.}
\newblock {\em {Indiana Univ. Math. J.}}, 35:45--71, 1986.

\bibitem{Ilmanen2}
T.~Ilmanen.
\newblock Singularities of the mean curvature flow of surfaces.
\newblock {\em preprint}, 1995.

\bibitem{Ilmanen}
T.~Ilmanen.
\newblock Lectures on mean curvature flow and related equations.
\newblock 1998.
\newblock http://www.math.ethz.ch/?ilmanen/papers/notes.ps.

\bibitem{KleeneMoller}
S.~J. Kleene and N.~M{\o}ller.
\newblock Self-shrinkers with a rotationnal symmetry.
\newblock {\em Trans. Amer. Math.Soc.}, 366(8):3943--3963, 2014.

\bibitem{LopezRos}
F.~J. L\'opez and A.~Ros.
\newblock On embedded minimal surfaces of genus zero.
\newblock {\em J. Differential Geom.}, 33(1):293--300, 1991.

\bibitem{Mantegazza}
C.~Mantegazza.
\newblock {\em Lecture notes on mean curvature flow}, volume 290.
\newblock Springer Basel, 2011.

\bibitem{MeeksRosenberg}
W.~H. {Meeks III} and H.~Rosenberg.
\newblock The uniqueness of the helicoid.
\newblock {\em Ann. of Math.}, 161:727--758, 2005.

\bibitem{KKM}
S.~Kleene N.~Kapouleas and N.M. M{\o}ller.
\newblock Mean curvature self-shrinkers of high genus: Non-compact examples.
\newblock arXiv:1106.5454, 2014.

\bibitem{Nguyen}
X.~H. Nguyen.
\newblock Construction of complete embedded self-similar surfaces under mean
  curvature flow. part ii.
\newblock arXiv:1106.5272, 2014.

\bibitem{Osserman}
R.~Osserman.
\newblock Global properties of minimal surfaces in ${E}^3$ and ${E}^n$.
\newblock {\em Ann. of Math.}, 80(2):340--364, 1964.

\bibitem{White}
B.~White.
\newblock Partial regularity of mean convex hypersurfaces flowing by mean
  curvature.
\newblock {\em announcement}, 1994.

\end{thebibliography}

\end{document}